\documentclass[11pt,a4paper]{article}
\usepackage{epsfig}
\usepackage{amsthm}

\usepackage[margin=1.1in]{geometry}
\usepackage{comment}
\usepackage{subfig}

\usepackage{amssymb,amsmath,latexsym,enumerate,graphicx,microtype}
\usepackage{abstract}
\usepackage{url}
\usepackage{hyperref}
\allowdisplaybreaks

\usepackage{amsmath,amssymb,amsthm,mathrsfs}
\usepackage{mathtools}
\usepackage{comment}
\usepackage{abstract}

\usepackage[colorinlistoftodos]{todonotes}

\renewcommand{\ge}{\geqslant}

\newcommand{\F}{{\cal{F}}} 
\newcommand{\A}{{\cal{A}}}

\newtheorem{theorem}{Theorem}
\newtheorem{corollary}{Corollary}
\newtheorem{conjecture}{Conjecture}
\newtheorem{lemma}{Lemma}
\newtheorem{problem}{Problem}
\newtheorem{prop}{Proposition}

\newtheorem{claim}{Claim}

\author{
	Eyal Ackerman\thanks{Department of Mathematics, Physics and Computer Science,
		University of Haifa at Oranim, 	Tivon 36006, Israel. \texttt{ackerman@math.haifa.ac.il}}\and
	G\'abor Dam\'asdi\thanks{Alfr\'ed R\'enyi Institute of Mathematics and ELTE E\H otv\H os Lor\'and University, Budapest, Hungary. Partially supported by ERC Advanced Grant GeoScape.
\texttt{damasdigabor@caesar.elte.hu}}\and
	Bal\'azs Keszegh\thanks{HUN-REN Alfréd Rényi Institute of Mathematics and ELTE Eötvös Loránd University, Budapest, Hungary. 
Research supported by the J\'anos Bolyai Research Scholarship of the Hungarian Academy of Sciences, by the National Research, Development and Innovation Office -- NKFIH under the grant K 132696 and FK 132060, by the \'UNKP-23-5 New National Excellence Program of the Ministry for Innovation and Technology from the source of the National Research, Development and Innovation Fund and by the ERC Advanced Grant ``ERMiD''. This research has been implemented with the support provided by the Ministry of Innovation and Technology of Hungary from the National Research, Development and Innovation Fund, financed under the  ELTE TKP 2021-NKTA-62 funding scheme.}\and
	Rom Pinchasi\thanks{Technion - Israel Institute of Technology, Haifa, Israel. \texttt{room@technion.ac.il}. Visiting professor at EPFL, Lausanne, Switzerland. Supported by ISF grant (grant No.\ 1091/21)}\and
	Rebeka Raffay\thanks{École Polytechnique Fédérale de Lausanne (EPFL), Lausanne, Switzerland. \texttt{ rebekaraffay@gmail.com}}
}

\title{The maximum number of digons formed by pairwise crossing pseudocircles}

\begin{document}
	\maketitle
	
	\begin{abstract}
    In 1972, Branko Grünbaum conjectured that any arrangement of $n>2$ pairwise crossing pseudocircles in the plane can have at most
    $2n-2$ digons (regions enclosed by exactly two pseudoarcs), with the bound being tight. 
    While this conjecture has been confirmed for 
    cylindrical arrangements of pseudocircles and more recently for geometric circles, we extend these results to any simple arrangement of pairwise intersecting pseudocircles. Using techniques from the above-mentioned special cases, we provide complete proof of Grünbaum's conjecture that has stood open for over five decades.
	\end{abstract}

	\section{Introduction}
In this paper we focus on families of pairwise crossing pseudocircles.  A \emph{family of pseudocircles} is a set $\F$ of simple closed Jordan curves in the plane such that every two of them are either disjoint, intersect at exactly one point in which they touch or intersect at exactly two
points in which they properly cross each other. 
The \emph{arrangement} $\A(\F)$ is the cell complex into which the plane is decomposed by the pseudocircles in $\F$. It consists of vertices, edges and faces. The arrangement is called \emph{trivial} if there are two points that lie on every pseudocircle in $\F$. If there is no point that lies on three pseudocircles, then both the
arrangement and $\F$ are called \emph{simple}.
Note that every simple arrangement of more than two pseudocircles is nontrivial. The arrangement is called \emph{pairwise intersecting} if either any two pseudocircles intersect in exactly two points or in exactly one point (where they touch). The arrangement is called \emph{pairwise crossing} if any two pseudocircles intersect in exactly two points.

A \emph{digon} is a face in the arrangement that is bounded by exactly two edges. If the two edges of a digon are contained in pseudocircles $c_1$ and $c_2$, respectively, then we say that $c_1$ and $c_2$ \emph{form} that digon and that each of $c_1$ and $c_2$ \emph{supports} it.

A trivial arrangement of $n$ pseudocircles contains $2n$ digons for $n>1$. In his 1972 monograph ``Arrangements and Spreads''~\cite{G72} Branko Gr\"unbaum conjectured that for nontrivial arrangements of pairwise crossing pseudocircles the maximum number of digons is $2n-2$. 
 
 \begin{conjecture}[Gr\"unbaum’s digon conjecture~{\cite[Conjecture 3.6]{G72}}]\label{conj:grunb}
	Every simple arrangement of $n > 2$ pairwise crossing pseudocircles has at most $2n-2$ digons.
\end{conjecture}

This conjecture was accompanied by a construction with exactly $2n-2$ digons, see Figure~\ref{fig:2n-2} for an example with six pseudocircles. This example can be easily extended by adding new pseudocircles to the four in the middle, which illustrates that the bound in  Conjecture \ref{conj:grunb} is tight. With some care, this construction can even be realized with geometric circles.
A different construction appears in~\cite{FRS23}.

\begin{figure}[!h]
	\centering
	\includegraphics[width=0.4\linewidth]{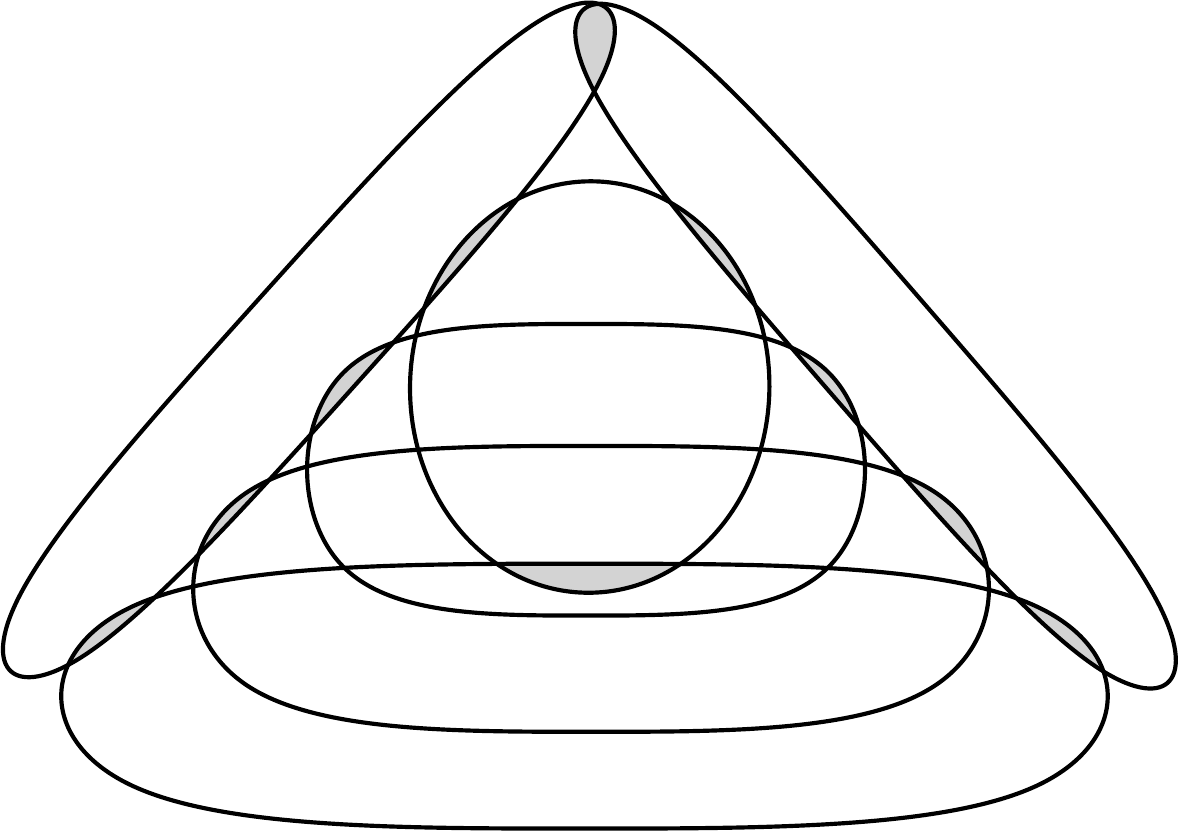}
	\caption{A family of 6 pseudocircles forming 10 digons.}
	\label{fig:2n-2}
\end{figure}

It is enough to consider \emph{simple} arrangements because it can be shown that small perturbations of the pseudocircles around the intersection points can be performed in a way that only increases or preserves the number of digons in nontrivial pairwise crossing arrangements. 


Conjecture \ref{conj:grunb} 
was shown to be true in several special cases. 
Agarwal et al.~\cite{ANPPSS04} showed that Conjecture \ref{conj:grunb} is true for \emph{cylindrical} arrangements, these are arrangements for which there is a point that is surrounded by each of the pseudocircles. 
More recently, Conjecture \ref{conj:grunb} was studied and advertised by Felsner, Roch and Scheucher~\cite{FRS23} who showed that it holds for arrangements in which there are three pseudocircles such that every two of them form a digon.
This new interest in Gr\"unbaum's conjecture has also
motivated us to study the problem. First, we were able to show in~\cite{circlepaper} that Gr\"unbaum's conjecture is true for simple arrangements of pairwise intersecting geometric circles. Here, using ideas from~\cite{circlepaper} and~\cite{ANPPSS04} we prove Gr\"unbaum's original conjecture.

\begin{theorem}\label{thm:main}
	Every simple arrangement of $n > 2$ pairwise crossing pseudocircles has at most $2n-2$ digons. This bound is tight.
\end{theorem}

Our proof is inspired by the proof technique introduced in \cite{ANPPSS04} for bounding the number of digons in cylindrical arrangements of $n$ pairwise intersecting pseudocircles.
In that paper, a bipartite graph is associated to a cylindrical arrangement in a way that each pseudocircle is represented by a single point which is its intersection with a common fixed transversal line.
The edges of this graph correspond to digons and are drawn according to a simple yet clever rule such that every two independent edges cross an even number of times. The Strong Hanani-Tutte Theorem (see below) then implies that the drawn (bipartite) graph is planar and hence by Euler's formula has at most $2n-4$ edges. 

For the general case of Gr\"unbaum's conjecture, we combine a modification of this graph drawing technique and a graph \emph{doubling} technique that was used for the case of geometric circles in~\cite{circlepaper}.
Namely, a bipartite graph is drawn such that every pseudocircle is represented by \emph{two} points which are its intersection with a common transversal pseudocircle.
Furthermore, every digon is represented by \emph{two} edges that are drawn following a rule in the spirit of~\cite{ANPPSS04} such that every two independent edges cross an even number of times.
Therefore, the drawn (bipartite) graph is planar and hence has at most $2(2n)-4=4n-4$ edges, representing at most $2n-2$ digons.

Note that our result can be applied for bounding the number of \emph{touching points} in simple arrangements where any two pseudocircle intersect or touch. In any such  arrangement of $n > 2$ pseudocircles, one can turn each touching point into exactly one digon between the two pseudocircles that support it. Therefore, for simple arrangements, an upper bound on the number of digons translates directly to a bound on the number of touching pairs or equivalently (in this case) touching points. One can also turn each digon into exactly one touching point in the case of simple arrangements if $n > 3$, but for $n = 3$, two pseudocircles may touch the third one at separate points forming a single digon (the region bounded by the entire third pseudocircle), see Figure \ref{fig:n_3}. In this scenario, turning the digon into a touching point would reduce the third pseudocircle to a single point, which is degenerate, and the final arrangement would violate the assumptions on the number of crossings. 

\begin{figure}[!h]
    \centering
    \includegraphics[width=0.5\linewidth]{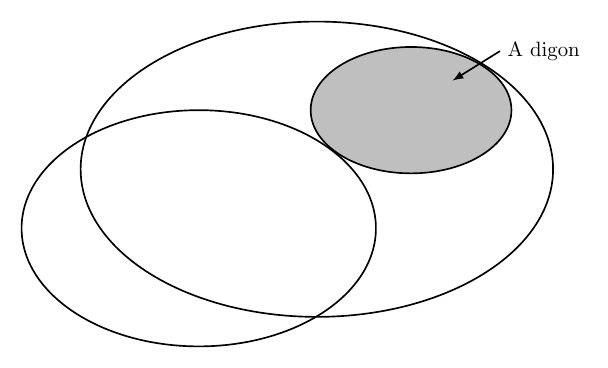}
    \caption{If we allow touchings and $n=3$, then a digon might be  surrounded by only one pseudocircle.}
    \label{fig:n_3}
\end{figure}

The transformation from digons to touching points or vice versa keeps their sum constant at every step (except for $n = 2$ or in the above-mentioned case), therefore they are equivalent for $n > 3$. From Theorem \ref{thm:main}, we have:
\begin{corollary}
Every simple pairwise intersecting arrangement of pseudocircles has at most $2n - 2$ touching pairs of pseudocircles. For $n > 3$, the sum of digons and touching pairs is at most $2n - 2$.
\end{corollary}

 

\subsection{{{Outline}}} In Section \ref{sec:tools} we collect useful observations and results that we need for the main proof. In Section \ref{sec:mainproof} we prove Theorem \ref{thm:main} and in Section \ref{sec:remarks} we discuss some open questions.

\section{Some terminology and tools}\label{sec:tools}

We begin with some simple definitions and observations, some of which also appear in~\cite{circlepaper} in the context of circles.

 Let $\F$ be a simple family of $n>2$ pairwise crossing pseudocircles in the plane. We may assume (and we will assume from now on) that every pseudocircle in $\F$ supports at least one digon. Otherwise, we can remove any pseudocircle from $\F$ that does not satisfy this condition and argue for the remaining family of pseudocircles. 
 
 As touchings are forbidden in a pairwise crossing family, each digon is supported by exactly two pseudocircles. A digon is called a \emph{lens} if it surrounded by both supporting pseudocircles, and it is called a \emph{lune} otherwise. A lune must be surrounded by precisely one of its supporting pseudocircles, while lying outside the region surrounded by the other (see Figure \ref{fig:lenlune}). Here we ignore the case where the unbounded face is a digon.
 This indeed can be avoided by means of a simple inversion of the plane.

\begin{figure}[!h]
    \centering
    \includegraphics[width=0.5\linewidth]{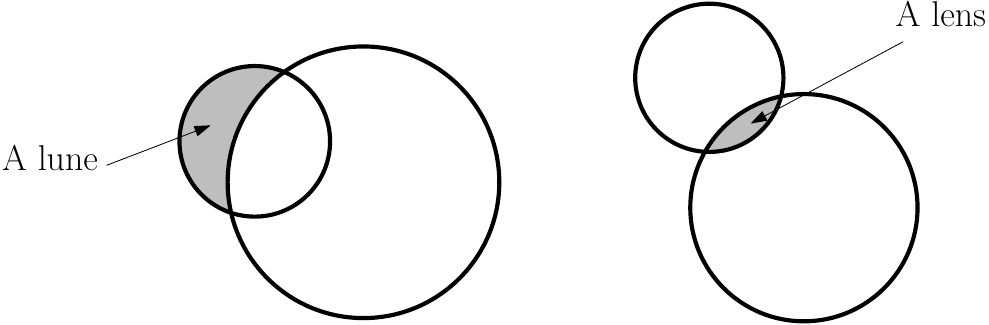}
    \caption{A lens and a lune.}
    \label{fig:lenlune}
\end{figure}

A pseudocircle $\alpha$ is \emph{internal} if it supports a digon that is surrounded by $\alpha$ and  it is called \emph{external} if it supports a digon (necessarily a
lune) that is not surrounded by $\alpha$. The following simple observation is important for the proof.

\begin{prop}\label{observation:int_ext}
A pseudocircle in $\F$ cannot be both internal and external.
\end{prop}

\begin{proof}
    
Assume to the contrary that there is a pseudocircle $\alpha \in \F$ which is both external and internal. This means that there exist $\alpha_1,\alpha_2\in \F$ such that $\alpha$ and $\alpha_1$ form a digon not surrounded by $\alpha$ while $\alpha$ and $\alpha_2$ form a digon that is surrounded by $\alpha$ (see Figure \ref{fig:int_ext}).

\begin{figure}[ht]
	\centering
	\includegraphics[height=4cm]{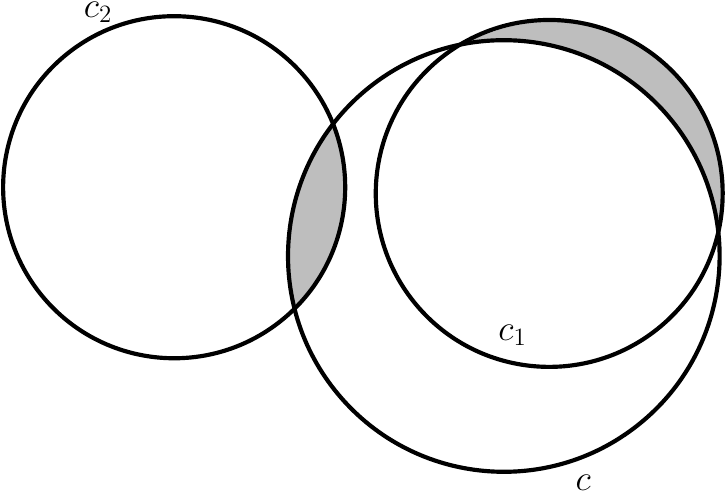}
	\caption{A pseudocircle cannot be both interior and exterior.}
	\label{fig:int_ext}
\end{figure}

In such a case $\alpha_1$ and $\alpha_2$ cannot intersect, contradicting our assumption that we have an arrangement of pairwise crossing pseudocircles. Indeed, $\alpha_{1}$ and $\alpha_{2}$ cannot cross in the region surrounded by $\alpha$ because that subarc of $\alpha_2$ is an edge of a digon. Similarly, $\alpha_{1}$ and $\alpha_{2}$ cannot intersect in the complementary region of the plane because that subarc of $\alpha_1$ is an edge of an other digon.
\end{proof}

Any simple closed Jordan curve in the plane divides the plane into two parts, one bounded (the \emph{interior}) and the other unbounded (the \emph{exterior}). We will refer to these two parts as the \emph{regions} of the curve. 
For a pseudocircle $\alpha \in \F$, we will call the region that contains the digons supported by $\alpha$ the \emph{digon-region} of $\alpha$. By Proposition~\ref{observation:int_ext} and because we assume that every pseudocircle supports at least one digon, every pseudocircle in $\F$ has exactly one digon-region. One  advantage of this definition is that we can treat lunes and lenses in a uniform way: every digon is just the intersection of the digon-regions of the two corresponding pseudocircles supporting it.\footnote{Another way to think about this is to imagine that we work on a sphere, and instead of circles we have caps corresponding to the digon-regions.} This will simplify the case-analysis in some of the proofs. 


\medskip
It is well-known that any family of pairwise crossing pseudocircles can be extended to a larger family of pairwise crossing pseudocircles, and the following lemma shows that we can also do it without destroying the existing digons.

\begin{lemma}\label{lem:base}
 
    Let $\F$ be a simple family of $n>2$ pairwise crossing pseudocircles. Then there is a closed curve $c$ such that $\F \cup \{c\}$ is a simple family of pairwise crossing pseudocircles, and furthermore $c$ does not intersect any of the digons in $\A(\F)$.
\end{lemma}
\begin{proof}
    Let $\alpha$ be a pseudocircle in $\F$. From Proposition~\ref{observation:int_ext} we know that one of the two regions of $\alpha$ (the one that is not the digon-region of $\alpha$) does not contain any digon supported by $\alpha$. Let $c$ be a curve that is running very close to $\alpha$ outside of the digon-region of $\alpha$. Then $c$ intersects each pseudocircle in $\F \setminus \{\alpha\}$ exactly twice and it intersects no digon of $\A(\F)$. 
    However, $c$ does not intersect $\alpha$. In order to fix this we notice that because $\F$ is a nontrivial arrangement of pseudocircles, there is a subarc $s$ of $\alpha$ that is disjoint from any digon supported by $\alpha$ and consequently it is disjoint from any digon in $\A(\F)$. We can modify $c$ along $s$ such that the resulting simple closed curve crosses $\alpha$ at two points yielding the desired extension of $\F$ (see Figure \ref{fig:canonincal}).
\end{proof}

 \begin{figure}[!h]
     \centering
     \includegraphics[width=6cm]{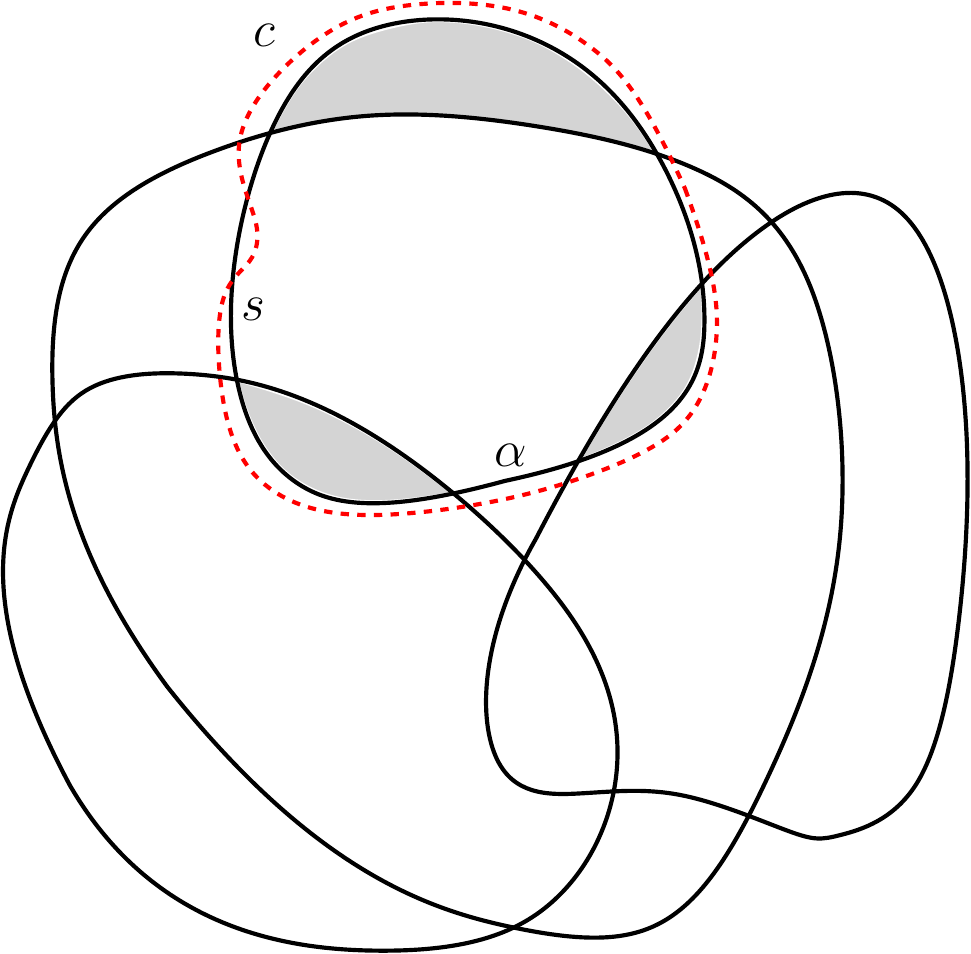}
     \caption{Extending $\F$ with a curve $c$ which avoids the digons in $\A(\F)$ and intersects every pseudocircle in $\F$ exactly twice.}
     \label{fig:canonincal}
 \end{figure}

For a pseudocircle $\alpha$ with a given orientation as a closed curve and points $v_1,v_2$ on $\alpha$ we denote by $[v_1,v_2]_\alpha$ (resp., $(v_1,v_2)_\alpha$) the closed (resp., open) segment of $\alpha$ from $v_1$ to $v_2$ following $\alpha$ in the given orientation. 

\begin{prop} \label{prop:cyclic-order}
	Let $\{\alpha_1,\alpha_2,c\}$ be a simple family of pairwise crossing pseudocircles  
	such that $\alpha_1$ and $\alpha_2$ form a digon.
	Suppose that $c$ is oriented and let $\alpha_j^{\rm in}$ (resp., $\alpha_j^{\rm out}$) be the intersection point of $c$ and $\alpha_j$ in which $c$ enters (resp., leaves) the digon-region of $\alpha_j$, for $j=1,2$.
	Then $\alpha_1^{\rm in}, \alpha_1^{\rm out}, \alpha_2^{\rm in}, \alpha_2^{\rm out}$ is the cyclic order along $c$ of the four intersection points of $\alpha_1$ and $\alpha_2$ with $c$.
\end{prop}

\begin{proof}
If $\alpha_2^{\rm in} \in (\alpha_1^{\rm in},\alpha_1^{\rm out})_c$ or $\alpha_2^{\rm out} \in (\alpha_1^{\rm in},\alpha_1^{\rm out})_c$, then we have a point from $\alpha_2 \cap c$ which lies in the digon-region of $\alpha_1$. 
However, the intersection of $\alpha_2$ with the digon-region of $\alpha_1$ consists only of the subarc of $\alpha_2$ which bounds the digon formed by $\alpha_1$ and $\alpha_2$.
This implies that $c$ intersects this digon which is impossible.
Therefore, $\alpha_1^{\rm out}$ must follow $\alpha_1^{\rm in}$ in the cyclic order and, similarly, $\alpha_2^{\rm out}$ must follow $\alpha_2^{\rm in}$.
\end{proof}


A \emph{topological graph} is a graph drawn in the plane such that its vertices are drawn as distinct points and its edges are drawn as Jordan arcs connecting the corresponding points.
Apart from its endpoints, an edge of a topological graph cannot contain any drawn vertex.
Furthermore, every two edges in a topological graph intersect at a finite number of points, each of which is either a common endpoint or a crossing point.
A graph is \emph{planar} if it can be drawn as a topological graph where no pair of edges crosses.
By the Strong Hanani-Tutte Theorem, it is enough to require an even number of crossings between independent edges.\footnote{Two edges are \emph{independent} if they do not share an endpoint.}



\begin{theorem}[Strong Hanani-Tutte Theorem~\cite{Tutte70}]\label{thm:strongHanani}
	A graph is planar if and only if it can be drawn as a topological graph in which every two independent edges cross an even number of times.
\end{theorem}


\section{Proof of Theorem \ref{thm:main}}\label{sec:mainproof}


Let $\F$ be a simple family of $n>2$ pairwise crossing pseudocircles in the plane such that every pseudocircle in $\F$ supports at least one digon.
By Lemma \ref{lem:base} there is a simple closed curve $c$ that intersects each pseudocircle in $\F$ twice but does not intersect the digons of the arrangement $\A(\F)$. 
We fix the counterclockwise orientation on $c$. Whenever we traverse $c$ or consider a cyclic order of points on it, we do it according to this orientation. Recall that $c$ is an auxiliary pseudocircle and it is not a member of $\F$; in particular we do not care about digons supported by $c$. 

Every pseudocircle $\alpha \in \F$ intersects $c$ at two points such that at one of them $c$ enters the digon-region of $\alpha$ and at the other $c$ leaves this digon-region. We denote these points by $\alpha^{\rm in}$ and $\alpha^{\rm out}$, respectively. See for example Figure \ref{fig:vertices_edges}. In this figure, the digon-regions of $\alpha$ and $\gamma$ are the interior regions of these curves, whereas the digon-region of $\beta$ is its exterior region.

\begin{figure}[!h]
    \centering
    \includegraphics[width=9cm]{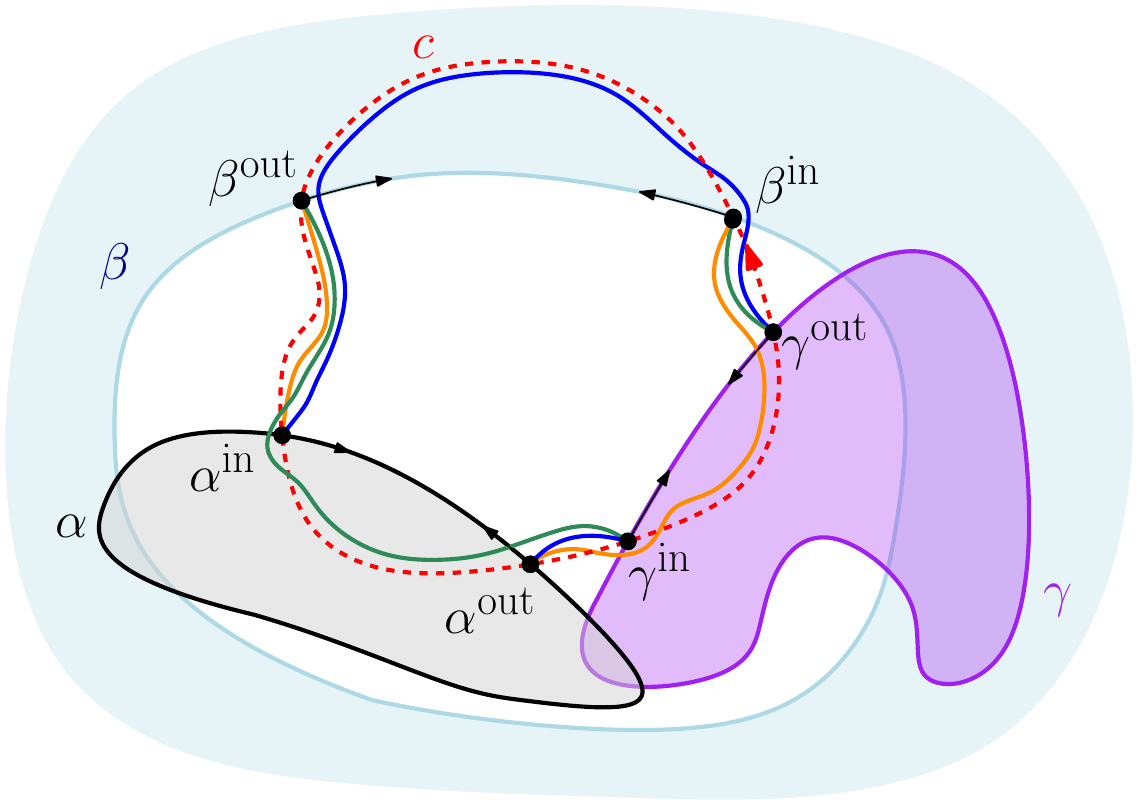}
    \caption{Three pseudocircles $\alpha$, $\beta$ and $\gamma$ and a drawing of the corresponding bipartite double cover graph $G$. For example, $\alpha$ and $\gamma$ form a digon, hence we draw the edges $(\alpha^{\rm out},\gamma^{\rm in})$ and $(\gamma^{\rm out},\alpha^{\rm in})$ along $c$. Considering the (blue) edge $(\gamma^{\rm out},\alpha^{\rm in})$, starting at $\beta^{\rm in}$ (resp., $\beta^{\rm out}$) and following $\beta$ towards the interior of $c$, we first encounter $\alpha$ (resp., $\gamma$). Therefore, the drawn edge $(\gamma^{\rm out},\alpha^{\rm in})$ ``passes by'' $\beta^{\rm in}$ (resp., $\beta^{\rm out}$) within the exterior (resp., interior) of $c$.} 
    \label{fig:vertices_edges}
\end{figure}

\subsection{The digon graph and its double cover}
In order to bound from above the number of digons in $\A(\F)$ we consider the graph $G_{d}$ whose 
vertex set corresponds to the pseudocircles in $\F$ and whose edge set corresponds to pairs of pseudocircles which form a digon in $\A(\F)$
(since $\A(\F)$ is nontrivial and crossing a pair of pseudocircles may not form more than one digon).
We call $G_{d}$ the \emph{digon graph} of the arrangement.
Bounding the number of edges in $G_{d}$ is equivalent to bounding the number of digons in $\A(\F)$.

For our proof, we consider the \emph{bipartite double cover} of $G_{d}$ - this graph, denote it by $G$, has two vertices, $\alpha^{\rm in}$ and $\alpha^{\rm out}$ for every vertex $\alpha$ of $G_d$, and two edges, $(\alpha^{\rm out},\beta^{\rm in})$ and $(\beta^{\rm out},\alpha^{\rm in})$ for every edge $(\alpha,\beta)$ of $G_d$.
Thus, the number of edges of $G$ is twice the number of digons in $\A(\F)$.
By construction $G$ is bipartite.
It remains to show that $G$ is also planar and therefore has at most $2|V(G)|-4=4n-4$ edges. This will prove Theorem~\ref{thm:main}.

\subsection{Drawing $G$ in the plane}
We draw $G$ as a topological graph in the plane.
By abuse of notation, we will not distinguish between $G$ and its drawing.
The vertices of $G$ are represented by the points $\{\alpha^{\rm in} \mid \alpha \in \F\}\cup\{\alpha^{\rm out} \mid \alpha\in \F\}$ as defined above. 
Thus, for every two pseudocircles $\alpha, \beta \in \F$ that form a digon in $\A(\F)$ there should be two drawn edges $(\alpha^{\rm out},\beta^{\rm in})$ and $(\beta^{\rm out},\alpha^{\rm in})$.
These edges will satisfy the following properties:
\begin{enumerate}
    \item[(i)] An edge $(\alpha^{\rm out},\beta^{\rm in})$ connects $\alpha^{\rm out}$ and $\beta^{\rm in}$, while avoiding all other vertices of $G$;
    \item[(ii)] every two edges intersect finitely many times; and
    \item[(iii)]  an edge $(\alpha^{\rm out},\beta^{\rm in})$ follows very closely the subarc of $c$ starting at $\alpha^{\rm out}$ and ending at $\beta^{\rm in}$, counterclockwise along $c$. It may cross the curve $c$ several times, but it does not intersect the digon-regions of neither $\alpha$ nor $\beta$.
\end{enumerate}

Properties (i), (ii), and (iii) can be easily satisfied.
For the second part of (iii) note that Proposition~\ref{prop:cyclic-order} implies that when following $c$ from $\alpha^{\rm out}$ to $\beta^{\rm in}$ we remain outside the digon-regions of $\alpha$ and $\beta$.

We need one more property that the drawing of $(\alpha^{\rm out},\beta^{\rm in})$
should satisfy. This is the most crucial property for the proof and it has to do with 
how $(\alpha^{\rm out},\beta^{\rm in})$ is drawn with respect to the intermediate
vertices of $G$ along the arc on $c$ from $\alpha^{\rm out}$ to $\beta^{\rm in}$. We use a modification of the drawing rule that appears in~\cite{ANPPSS04}.



Let $v \in V$ be an intermediate vertex of $G$ on $c$ along the arc from $\alpha^{\rm out}$
to $\beta^{\rm in}$. We will need to indicate whether the edge $(\alpha^{\rm out},\beta^{\rm in})$
passes near $v$ in the interior region of $c$ or in its exterior region.
In the first case, we say that the edge \emph{passes $v$ from the inside} whereas in the second case it \emph{passes $v$ from the outside}. Other than these inside/outside decisions and properties (i)-(iii), $(\alpha^{\rm out},\beta^{\rm in})$ can be drawn arbitrarily.

We determine whether $(\alpha^{\rm out},\beta^{\rm in})$ passes $v$ from the inside or 
from the outside according to the value $d(v,(\alpha^{\rm out},\beta^{\rm in}))$ that we define in the following way.
Let $\gamma$ be the pseudocircle in $\F$ that intersects $c$ at $v$. It follows from Proposition~\ref{prop:cyclic-order} that $\gamma$ must be different from $\alpha$ and $\beta$. We follow the curve $\gamma$ starting from $v$ in the direction \emph{entering the interior region of $c$}. 
If we intersect $\alpha$ sooner than we intersect $\beta$, then $d(v,(\alpha^{\rm out},\beta^{\rm in}))=1$, otherwise $d(v,(\alpha^{\rm out},\beta^{\rm in}))=2$.\footnote{Thus, the value $1$ or $2$ represents whether we encounter the first or the second pseudocircle among the two pseudocircles of the ordered pair $(\alpha^{\rm out}, \beta^{\rm in})$.} Since $\F$ is a simple family of pairwise crossing curves,  $d(v,(\alpha^{\rm out},\beta^{\rm in}))$ is well-defined.
If $d(v,(\alpha^{\rm out},\beta^{\rm in}))=1$, then we pass $v$ from the inside. Otherwise $d(v,(\alpha^{\rm out},\beta^{\rm in}))=2$ and we pass $v$ from the outside.
See Figure \ref{fig:vertices_edges} for examples. 



Having described the drawing of the graph $G$ in the plane, we will conclude the proof by showing that every two independent edges of $G$ cross an even number of times which implies that $G$ is planar by the Strong Hanani-Tutte Theorem (Theorem \ref{thm:strongHanani}).

\subsection{Independent edges cross evenly}

Let $(\alpha^{\rm out},\beta^{\rm in})$ and $(\gamma^{\rm out},\delta^{\rm in})$ be two independent edges of $G$.
The parity of the number of crossings between two edges depends solely on how one edge passes by a vertex of the other edge.
For example, if the cyclic order of their endpoints is $\alpha^{\rm out}, \gamma^{\rm out}, \delta^{\rm in}, \beta^{\rm in}$, then the two edges cross evenly if and only if $(\alpha^{\rm out},\beta^{\rm in})$ passes by both of $\gamma^{\rm out}$ and $\delta^{\rm in}$ from the inside or both of them from the outside.
Therefore, it is enough to consider all the possible topologically different arrangements of the at most four involved pseudocircles (actually five because of $c$) and verify that in each of them $(\alpha^{\rm out},\beta^{\rm in})$ and $(\gamma^{\rm out},\delta^{\rm in})$ cross an even number of times.
Using a computer we have verified this, still, for the remainder of the section (and of the proof) we provide a non-computer-aided proof.

\medskip
We begin by observing that it is enough to consider the case that the involved pseudocircles $\alpha$, $\beta$, $\gamma$, and $\delta$ are distinct.
Indeed, $(\alpha^{\rm out},\beta^{\rm in})$ and $(\gamma^{\rm out},\delta^{\rm in})$ are independent edges, therefore, $\alpha \ne \gamma$ and $\beta \ne \delta$.
Still, it is possible that $\alpha=\delta$ or $\beta=\gamma$. If $\alpha = \delta$ and $\beta = \gamma$ then by Proposition \ref{prop:cyclic-order} the edges do not cross.
Otherwise, we add a new pseudocircle as a substitute in the following way.
Suppose without loss of generality that $\beta=\gamma$ (if $\alpha=\delta$ we switch the labels $\alpha$ and $\gamma$ and the labels $\beta$ and $\delta$).
Then we add a new pseudocircle $\gamma'$ that goes very close to $\beta=\gamma$
outside of the digon-region of $\beta=\gamma$ except for near the vertices of the digon formed by 
$\beta=\gamma$ and $\delta$. There, the pseudocircle $\gamma'$ goes into the digon-region of $\beta=\gamma$
thus destroying the digon formed by $\beta=\gamma$ and $\delta$ and creating a new digon 
with $\delta$ that is fully contained in the digon that was destroyed.
See Figure~\ref{fig:gamma_prime} for an example.

\begin{figure}
    \centering
    \includegraphics[width=8cm]{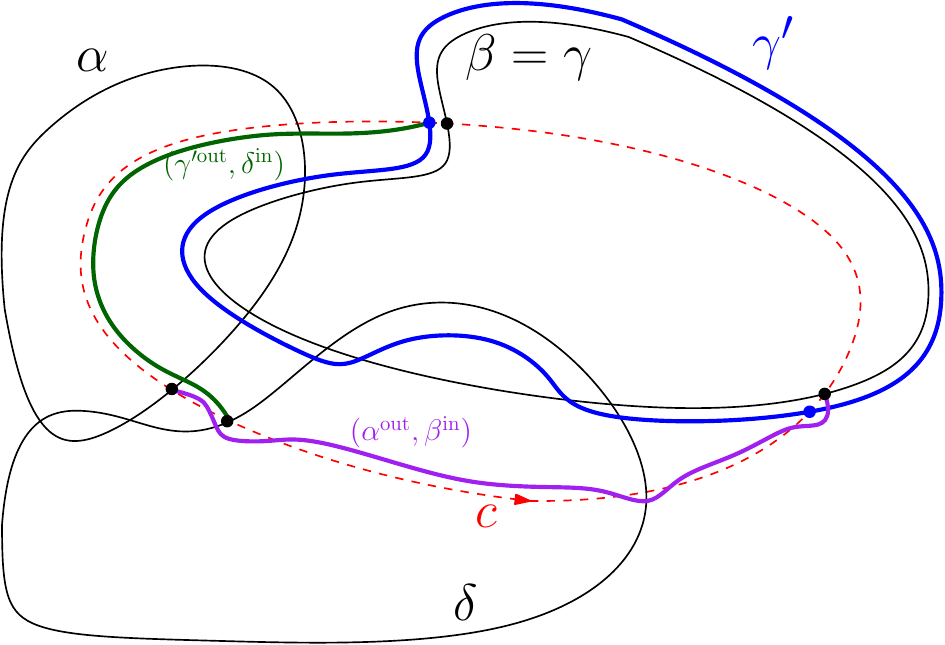}
    \caption{If $(\alpha^{\rm out},\beta^{\rm in})$ and $(\gamma^{\rm out},\delta^{\rm in})$ are edges and $\beta=\gamma$, then we can replace $\gamma$ with $\gamma'$ such that $(\alpha^{\rm out},\beta^{\rm in})$ and $(\gamma'^{\rm out},\delta^{\rm in})$ maintain the same number of crossings as the original edges.}
    \label{fig:gamma_prime}
\end{figure}

The edge $(\gamma'^{\rm out},\delta^{\rm in})$ starts at $\gamma'^{\rm out}$ and follows the edge $(\gamma^{\rm out},\delta^{\rm in})$, while $(\alpha^{\rm out},\beta^{\rm in})$ passes by $\gamma'^{\rm in}$ according the above rules.
The number of crossings of the two independent edges $(\alpha^{\rm out},\beta^{\rm in})$ and $(\gamma'^{\rm out},\delta^{\rm in})$ is equal to that of the original edges 
because $(\gamma'^{\rm out},\delta^{\rm in})$ and $(\gamma^{\rm out},\delta^{\rm in})$ are drawn in precisely the same way 
as far as passing from the inside or from the outside of intermediate vertices of $G$.
Note that the new vertex $\gamma'^{\rm in}$ becomes an intermediate vertex for $(\alpha^{\rm out}, \beta^{\rm in})$, however, it is not an endpoint of $(\gamma'^{\rm out},\delta^{\rm in})$, thus it does not matter whether  $(\alpha^{\rm out}, \beta^{\rm in})$ passes it from the inside or from the outside.

\medskip
Henceforth we assume that the pseudocircles $\alpha$, $\beta$, $\gamma$ and $\delta$ are distinct. 
We proceed by considering the six possible cyclic orders of the vertices $\alpha^{\rm out}$, $\beta^{\rm in}$, $\gamma^{\rm out}$ and $\delta^{\rm in}$.
If the cyclic order is $(\alpha^{\rm out}, \beta^{\rm in}, \gamma^{\rm out}, \delta^{\rm in})$, then $(\alpha^{\rm out},\beta^{\rm in})$ and $(\gamma^{\rm out},\delta^{\rm in})$ do not intersect.
The cyclic orders $(\alpha^{\rm out},\gamma^{\rm out},\beta^{\rm in},\delta^{\rm in})$ and $(\alpha^{\rm out},\delta^{\rm in},\beta^{\rm in},\gamma^{\rm out})$ are symmetric (by switching the labels $\alpha$ and $\gamma$ and the labels $\beta$ and $\delta$), and the same goes for the cyclic orders $(\alpha^{\rm out},\beta^{\rm in},\delta^{\rm in},\gamma^{\rm out})$ and $(\alpha^{\rm out},\gamma^{\rm out},\delta^{\rm in},\beta^{\rm in})$.
Therefore, it is enough to consider the cyclic orders $(\alpha^{\rm out},\gamma^{\rm out},\beta^{\rm in},\delta^{\rm in})$, $(\alpha^{\rm out},\gamma^{\rm out},\delta^{\rm in},\beta^{\rm in})$ and $(\alpha^{\rm out},\delta^{\rm in},\gamma^{\rm out},\beta^{\rm in})$.

Recall that we only care about how each edge passes by the vertices of the other edge. 
The following proposition summarizes the desired properties. 



\begin{prop}\label{prop:conditions}

  Assume that the following conditions are satisfied: 
    \begin{enumerate}
        \item[(A)]  If $\alpha^{\rm out}, \gamma^{\rm out}, \beta^{\rm in}$ and $\delta^{\rm in}$, appear in this cyclic order along $c$, then $d(\gamma^{\rm out},(\alpha^{\rm out},\beta^{\rm in}))\neq d(\beta^{\rm in},(\gamma^{\rm out},\delta^{\rm in}))$;
        \item[(B)]  if $\alpha^{\rm out}, \gamma^{\rm out}, \delta^{\rm in}$ and $\beta^{\rm in}$ appear in this cyclic order along $c$, then $d(\gamma^{\rm out},(\alpha^{\rm out},\beta^{\rm in}))= d(\delta^{\rm in},(\alpha^{\rm out},\beta^{\rm in}))$; and
        \item[(C)]  if $\alpha^{\rm out}, \delta^{\rm in}, \gamma^{\rm out}$ and $\beta^{\rm in}$ appear in this cyclic order along $c$, then $d(\alpha^{\rm out},(\gamma^{\rm out},\delta^{\rm in}))\neq d(\delta^{\rm in},(\alpha^{\rm out},\beta^{\rm in}))$ and $d(\gamma^{\rm out},(\alpha^{\rm out},\beta^{\rm in}))\neq d(\beta^{\rm in},(\gamma^{\rm out},\delta^{\rm in}))$.      
    \end{enumerate}
	Then $(\alpha^{\rm out},\beta^{\rm in})$ and $(\gamma^{\rm out},\delta^{\rm in})$ cross an even number of times.
\end{prop}


\begin{proof}
  From properties (i), (ii) and (iii) it follows that the parity of the number of crossings between $(\alpha^{\rm out},\beta^{\rm in})$ and $(\gamma^{\rm out},\delta^{\rm in})$ depends only on the inside/outside decisions at the vertices $\alpha^{\rm out}, \beta^{\rm in}, \gamma^{\rm out}$ and $\delta^{\rm in}$. This means that it is enough to check a single drawing for each set of inside/outside decisions that are allowed by the claim. 
  The proof follows by a case analysis that is easy to verify by inspection,
  see Figure~\ref{fig:edge_pairs} for illustrations of the possible cases.
  %
 %
%
%
\begin{figure}[t]
	\centering
	\subfloat[The cyclic order $(\alpha^{\rm out},\gamma^{\rm out},\beta^{\rm in},\delta^{\rm in})$ implies $d(\gamma^{\rm out},(\alpha^{\rm out},\beta^{\rm in})) \ne d(\beta^{\rm in},(\gamma^{\rm out},\delta^{\rm in}))$.]{\includegraphics[width= 6.65cm]{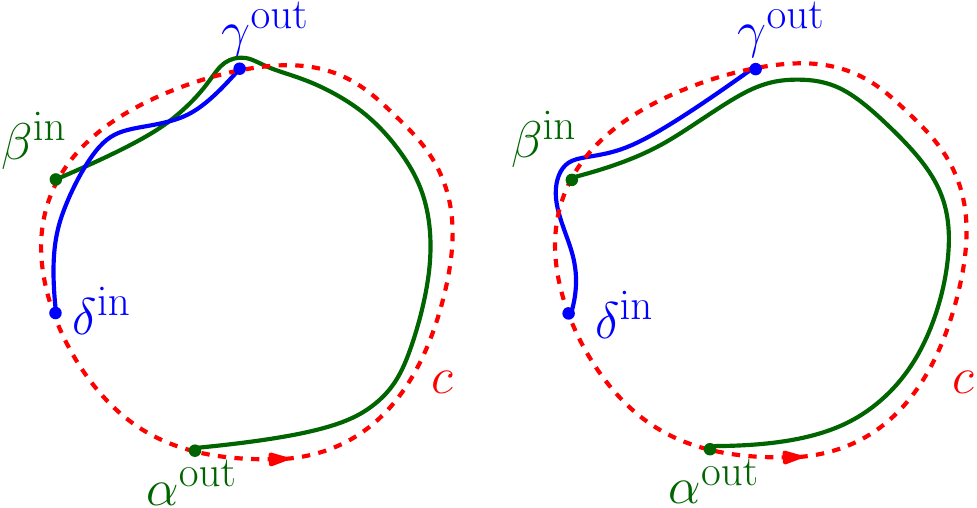}\label{fig:orderA}}
	\hspace{5mm}
	\centering
\subfloat[The cyclic order $(\alpha^{\rm out},\gamma^{\rm out},\delta^{\rm in},\beta^{\rm in})$ implies $d(\gamma^{\rm out},(\alpha^{\rm out},\beta^{\rm in}))= d(\delta^{\rm in},(\alpha^{\rm out},\beta^{\rm in}))$.]{\includegraphics[width= 6.65cm]{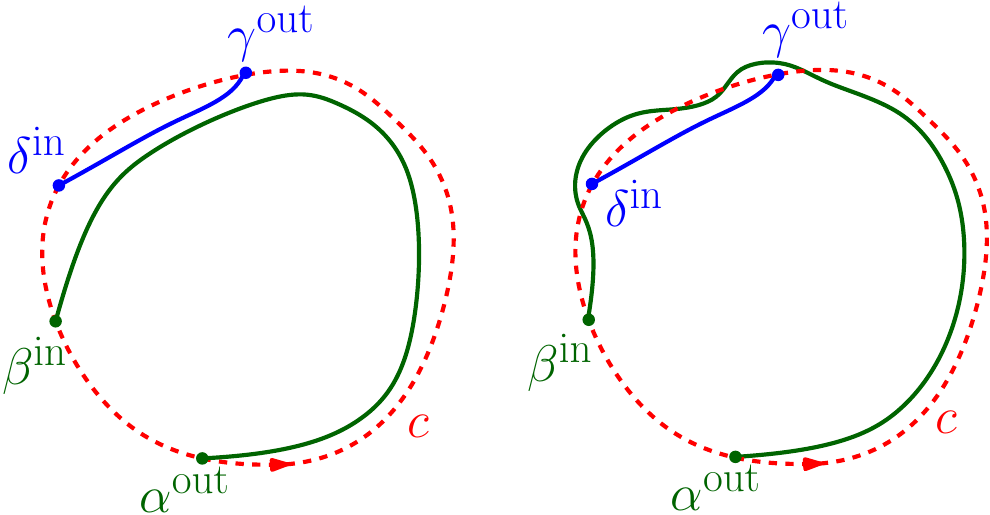}\label{fig:orderB}}
\hspace{5mm}
\subfloat[The cyclic order $(\alpha^{\rm out},\delta^{\rm in},\gamma^{\rm out},\beta^{\rm in})$ implies $d(\alpha^{\rm out},(\gamma^{\rm out},\delta^{\rm in}))\neq d(\delta^{\rm in},(\alpha^{\rm out},\beta^{\rm in}))$ and $d(\gamma^{\rm out},(\alpha^{\rm out},\beta^{\rm in}))\neq d(\beta^{\rm in},(\gamma^{\rm out},\delta^{\rm in}))$.]{\includegraphics[width= 14cm]{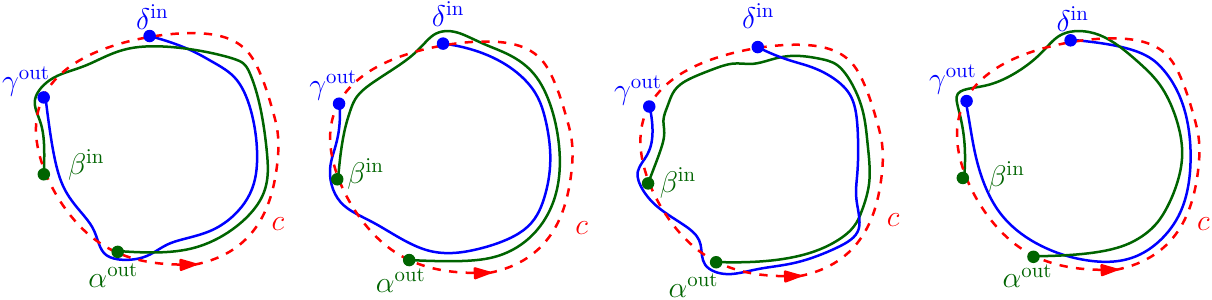}\label{fig:orderC}}
	\caption{$(\alpha^{\rm out},\beta^{\rm in})$ and $(\gamma^{\rm out},\delta^{\rm in})$ cross evenly if the conditions of Proposition~\ref{prop:conditions} hold.} 
	\label{fig:edge_pairs}
\end{figure}
\end{proof}

It remains to prove that the drawing of $G$ satisfies the conditions stated in Proposition~\ref{prop:conditions}.
Conditions (A) and (C) are shown to hold in Proposition~\ref{prop:conditions-A-C}
while (B) is handled in Proposition~\ref{prop:condition-B}.
Note that in Condition~(C) there is a symmetry between the two edges, namely, the cyclic order can also be written as $(\gamma^{\rm out}, \beta^{\rm in}, \alpha^{\rm out}, \delta^{\rm in})$. Therefore it is enough to prove, say, that $d(\gamma^{\rm out},(\alpha^{\rm out},\beta^{\rm in}))\neq d(\beta^{\rm in},(\gamma^{\rm out},\delta^{\rm in}))$ holds for the given cyclic order of the endpoints.




\begin{prop}\label{prop:conditions-A-C}
    If the cyclic of order of $\alpha^{\rm out}, \beta^{\rm in}, \gamma^{\rm out}$ and $\delta^{\rm in}$ along $c$ is either $(\alpha^{\rm out}, \delta^{\rm in}, \gamma^{\rm out}, \beta^{\rm in})$ or $( \alpha^{\rm out}, \gamma^{\rm out}, \beta^{\rm in},\delta^{\rm in})$, then $d(\gamma^{\rm out},(\alpha^{\rm out},\beta^{\rm in}))\neq d(\beta^{\rm in},(\gamma^{\rm out},\delta^{\rm in}))$.
\end{prop}
\begin{proof}
Orient $\beta$ and $ \gamma$ such that they enter the interior of $c$ at $\beta^{\rm in}$ and $\gamma^{\rm out}$, respectively (recall that $\beta \ne \gamma$). 
%
%
%
%
    We will show that if $d(\gamma^{\rm out},(\alpha^{\rm out},\beta^{\rm in}))=2$, then $d(\beta^{\rm in},(\gamma^{\rm out},\delta^{\rm in}))=1$. A symmetric argument (reflect the arrangement and switch the role of $\beta$ and $\gamma$ and also of $\alpha$ and $\delta$) shows that if $d(\beta^{\rm in},(\gamma^{\rm out},\delta^{\rm in}))=1$, then $d(\gamma^{\rm out},(\alpha^{\rm out},\beta^{\rm in}))=2$. Together, these two assertions imply the claim. 

    Suppose that $d(\gamma^{\rm out},(\alpha^{\rm out},\beta^{\rm in}))=2$. This means 
    that as we follow $\gamma$ starting from $\gamma^{\rm out}$ it intersects $\beta$ before it intersects $\alpha$. Let $x$ denote this first intersection point of $\gamma$ and $\beta$ (see Figure \ref{fig:xyb1a2}). 
    We will show that $[\beta^{\rm in},x]_{\beta}$ does not intersect $\delta$, and hence $d(\beta^{\rm in},(\gamma^{\rm out},\delta^{\rm in}))=1$.

Let $t$ be the closed curve which consists of $[\gamma^{\rm out},x]_{\gamma}$, $[\beta^{\rm in},x]_{\beta}$ and $[\gamma^{\rm out},\beta^{\rm in}]_{c}$.
We claim that $t$ is a simple curve that is not crossed by any of $\alpha$, $\beta$, and $\gamma$.

Consider first $\beta$ and observe that it does not cross itself nor can it cross $[\gamma^{\rm out},x]_\gamma$ by the definition of $x$.
From Proposition~\ref{prop:cyclic-order}, it follows that $\beta^{\rm out} \in (\beta^{\rm in},\alpha^{\rm out})_c$ and therefore $\beta$ does not cross $[\gamma^{\rm out},\beta^{\rm in}]_c$, and hence it does not cross $t$.
	
Next, consider $\gamma$ and observe that it does not cross itself nor can it cross $[\gamma^{\rm out},\beta^{\rm in}]_c$ since it follows from Proposition~\ref{prop:cyclic-order} that $\gamma^{\rm in} \in (\delta^{\rm in},\gamma^{\rm out})_c$.
Therefore, $t$ is a simple curve (recall that $\beta$ and $(\gamma^{\rm out},x)_\gamma$ cannot intersect by the definition of $x$).
Let $T$ be the `triangular' region which is bounded by $t$ and is to the right of $\gamma$ (and to the left of $\beta$ and $c$).

Returning to $\gamma$, it remains to show that it cannot intersect $(\beta^{\rm in},x)_\beta$.
Suppose for contradiction that it does intersect $(\beta^{\rm in},x)_\beta$ at a point $y$.
Then either $\gamma$ enters $T$ at $y$ or it leaves $T$ at $y$.
In the first case, following $\gamma$ from $y$, it must leave $T$ since the part of $\gamma$ just before $\gamma^{\rm out}$ lies outside $T$.
However, this means that $\gamma$ must cross $t$ and hence $(\beta^{\rm in},x)_\beta$ at another point which is impossible since $\beta$ and $\gamma$ already intersect at $x$ and $y$, see Figure~\ref{fig:AandC-2}.
    \begin{figure}[!h]
	\centering
	\subfloat[If $\gamma$ enters $T$ at $y$, then it must intersect {$(\beta^{\rm in},x)_\beta$} once more.]{\includegraphics[width= 5cm]{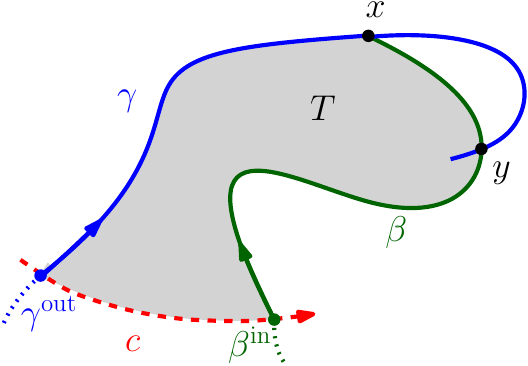}\label{fig:AandC-2}}
	\hspace{5mm}
	\subfloat[If $\gamma$ leaves $T$ at $y$, then $\beta$ enters $T$ at $x$ and must cross $t$ at some other point.]{\includegraphics[width= 5cm]{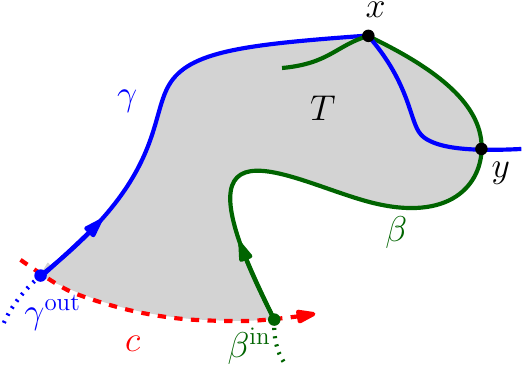}\label{fig:AandC-3}}
	\caption{Illustrations for the proof of Proposition~\ref{prop:conditions-A-C}: when following $\gamma$ into the interior of $c$ we meet $\beta$ sooner than $\alpha$ at a point $x$. Suppose for contradiction that $\gamma$ intersects $[\beta^{\rm in},x]_\beta$ at another point $y$.}
	\label{fig:xyb1a2}
\end{figure}
Considering the second case, if $\gamma$ leaves $T$ at $y$, then it implies that $\gamma$ enters $T$ at $x$, which in turn implies that $\beta$ is also entering $T$ at $x$ (otherwise $\beta$ and $\gamma$ would touch at $x$ which is impossible), see Figure~\ref{fig:AandC-3}.
However, following $\beta$ from $x$, we would have to cross $t$ since the part of $\beta$ just before $\beta^{\rm in}$ lies outside $T$.
Since we have already observed that $\beta$ does not cross $t$ we conclude that $\gamma$ does not cross $t$ either.

We now show that $\alpha$ cannot intersect $t$. 
By the definition of $x$ it cannot intersect $[\gamma^{\rm out},x]_{\gamma}$.
By Proposition~\ref{prop:cyclic-order} we have $\alpha^{\rm in} \in (\beta^{\rm in},\alpha^{\rm out})_c$ 
and it follows that $\alpha$ cannot intersect $[\gamma^{\rm out},\beta^{\rm in}]_{c}$.
Finally, if $\alpha$ crosses $[\beta^{\rm in},x]_{\beta}$, then it must cross it twice since $\alpha$ does not cross the other parts of $t$.
The digon formed by $\alpha$ and $\beta$ must lie outside of $T$ since $T$ lies outside the digon-region of $\beta$.
Consequently, apart from its subarc that bounds this digon $\alpha$ is contained in $T$.
However, this is impossible since then $\gamma$ and $\alpha$ do not intersect.

We are now ready to show that $\delta$ cannot intersect $[\beta^{\rm in},x]_{\beta}$ which would imply that $d(\beta^{\rm in},(\gamma^{\rm out},\delta^{\rm in}))=1$.
Assume to the contrary that $\delta$ intersects $[\beta^{\rm in},x]_{\beta}$. 
Notice that $\delta^{\rm in}, \delta^{\rm out}, \gamma^{\rm out}$, and
$\beta^{\rm in}$ appear in this cyclic order on $c$. This follows from Proposition~\ref{prop:cyclic-order} and the assumptions of the claim.
This implies that $\delta$ does not intersect $[\gamma^{\rm out},\beta^{\rm in}]_c$.
It follows that the two simple closed curves $\delta$ and $t$ must either cross twice or four times, since crossing six times or more would imply that $\delta$ intersects $\gamma$ or $\beta$ more than twice.
If $\delta$ and $t$ cross four times, then $\delta$ crosses twice each of $[\beta^{\rm in},x]_{\beta}$ and $[\gamma^{\rm out},x]_{\gamma}$. 
Therefore, the only part of $\delta$
that is outside the digon-region of $\gamma$ and the digon-region of $\beta$ must be in $T$.
The curves $\alpha$ and $\delta$ cannot cross in the digon-region of $\gamma$, because $\delta$ and $\gamma$ form a digon. Similarly, $\alpha$ and $\delta$ cannot cross in the digon-region of $\beta$ because $\alpha$ and $\beta$ form a digon. Since $\alpha$ is disjoint from $T$ we conclude that $\alpha$ and $\delta$ do not intersect which is a contradiction.
We reach a similar contradiction in the case where $\delta$ crosses $t$ precisely at two points that happen to be on $[\beta^{\rm in},x]_{\beta}$, once again $\alpha$ and $\delta$ cannot intersect.

It remains to consider the case where $\delta$ crosses $[\beta^{\rm in},x]_{\beta}$ once
and crosses $[\gamma^{\rm out},x]_{\gamma}$ once. 
The latter implies that either $x$ or $\gamma^{\rm out}$ belong to the digon formed by $\delta$ and $\gamma$, which is impossible since neither $\beta$ nor $c$ intersects this digon.
\end{proof}


\begin{prop}\label{prop:condition-B}
If $\alpha^{\rm out}, \gamma^{\rm out}, \delta^{\rm in}$ and  $\beta^{\rm in}$ appear in this cyclic order along $c$, then we have $d(\gamma^{\rm out},(\alpha^{\rm out},\beta^{\rm in}))= d(\delta^{\rm in},(\alpha^{\rm out},\beta^{\rm in}))$.
\end{prop}

\begin{proof}
Orient $\gamma$ and $\delta$ such that they enter the interior region of $c$ at $\gamma^{\rm out}$ and $\delta^{\rm in}$, respectively.
Recall that $c$ is oriented counterclockwise and its interior region lies to its left.
Thus, the digon-region of $\gamma$ lies to its left whereas the digon-region of $\delta$ lies to its right.
Let $\{A'_1,A'_2\} = \delta \cap \alpha$ and $\{B'_1,B'_2\} = \delta \cap \beta$, such that the  cyclic order of these four points along $\delta$ is $A'_1,A'_2,B'_1,B'_2$ (since $\alpha$ and $\beta$ form a digon it follows from Proposition~\ref{prop:cyclic-order} that $A_1'$ and $A'_2$ must be consecutive and so are $B'_1$ and $B'_2$).
Similarly, let $\{A''_1,A''_2\} = \gamma \cap \alpha$ and $\{B''_1,B''_2\} = \gamma \cap \beta$, such that the cyclic order of these four points along $\gamma$ is $A''_1,A''_2,B''_1,B''_2$.

Considering the cyclic order of $A'_1,A'_2,A''_1$ and $A''_2$ along $\alpha$ (under some orientation of $\alpha$), it follows
from Proposition~\ref{prop:cyclic-order} that the first two are consecutive and so are the last two (since $\gamma$ and $\delta$ form a digon).
Therefore, if we orient $\alpha$ such that $A'_2$ follows $A'_1$, then the cyclic order of those four points is either $A'_1,A'_2,A''_1,A''_2$ or $A'_1,A'_2,A''_2,A''_1$.

We claim that the former is impossible.
Indeed, suppose that $A'_1,A'_2,A''_1$ and $A''_2$ appear in this cyclic order on $\alpha$.
Note that it follows from the cyclic order on $\delta$ (resp., $\gamma$) that both of $B'_1$ and $B'_2$ (resp., $B''_1$ and $B''_2$) must be in the same region (interior or exterior) of $\alpha$.
If $B'_1, B'_2, B''_1$ and $B''_2$ are in the exterior of $\alpha$, then 
either $[A''_1,A''_2]_\alpha$ or $[A_1',A'_2]_\alpha$ is contained in the intersection of the digon-regions of $\gamma$ and $\delta$ (see Figure~\ref{fig:A'_1A'_2A''_1A''2}), which is impossible.
\begin{figure}[t]
	\centering
	\includegraphics[width= 10cm]{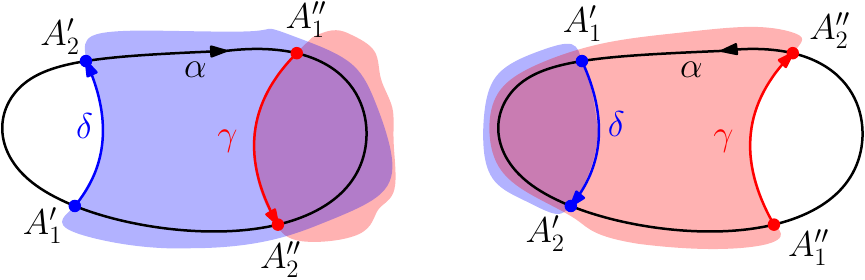}
	\caption{Suppose that $A'_1,A'_2,A''_1,A''_2$ appear in this cyclic order on $\alpha$ and $B'_1, B'_2, B''_1$ and $B''_2$ are in the exterior of $\alpha$. Then the intersection of the digon-region of $\delta$ (which is to its right) and the digon-region of $\gamma$ (which is to its left) contains either $[A''_1,A''_2]_\alpha$ or $[A'_1,A'_2]_\alpha$ which is impossible.}
	\label{fig:A'_1A'_2A''_1A''2}
\end{figure}
By a symmetric argument, it is impossible that $B'_1, B'_2, B''_1$, and $B''_2$ are in the interior of $\alpha$.
Therefore, either $B'_1$ and $B'_2$ are in the digon-region of $\alpha$ or $B''_1$ and $B''_2$ are there. However, this implies that the subarc of $\beta$ in the digon-region of $\alpha$ intersects either $\gamma$ or $\delta$ which is impossible since $\alpha$ and $\beta$ form a digon.

Thus $A'_1,A'_2,A''_2,A''_1$ appear in this cyclic order on $\alpha$, and similarly, $B'_1,B'_2,B''_2,B''_1$ appear in this cyclic order on $\beta$ if it is oriented such that $B'_2$ immediately follows $B'_1$.
Next, we redraw $\alpha$ close to the digon it forms with $\beta$ such that these two pseudocircles become disjoint (see Figure \ref{fig:arrangement-of-4}).
In a similar way, we redraw $\delta$ such that $\delta$ and $\gamma$ become disjoint.
Note that no other intersection points but $\alpha \cap \beta$ and $\gamma \cap \delta$ are destroyed and no new intersection points are introduced.
Therefore, the vertex set of $\A(\{\alpha,\beta,\gamma,\delta\})$ is precisely  $\{A'_1,A'_2,A''_1,A''_2,B'_1,B'_2,B''_1,B''_2\}$ and its edge set consists of the edges
$[A'_1,A'_2]_\alpha$, $[A'_2,A''_2]_\alpha$, $[A''_2,A''_1]_\alpha$, $[A''_1,A'_1]_\alpha$, 
$[B'_1,B'_2]_\beta$, $[B'_2,B''_2]_\beta$, $[B''_2,B''_1]_\beta$, $[B''_1,B'_1]_\beta$, 
$[A''_1,A''_2]_\gamma$, $[A''_2,B''_1]_\gamma$, $[B''_1,B''_2]_\gamma$, $[B''_2,A''_1]_\gamma$, 
$[A'_1,A'_2]_\delta$, $[A'_2,B'_1]_\delta$, $[B'_1,B'_2]_\delta$ and $[B'_2,A'_1]_\delta$.
It is not hard to see that this implies that the face set of $\A(\{\alpha,\beta,\gamma,\delta\})$ consists of four digons whose boundaries are $\{[A'_1,A'_2]_\alpha,[A'_1,A'_2]_\delta\}$, $\{[A''_1,A''_2]_\alpha,[A''_1,A''_2]_\gamma\}$, $\{[B'_1,B'_2]_\beta,[B'_1,B'_2]_\delta\}$
and $\{[B''_1,B''_2]_\beta,[B''_1,B''_2]_\gamma\}$, and six faces of size four whose boundaries are (refer to Figure~\ref{fig:arrangement-of-4} for an illustration): \\
$\{[A'_1,A'_2]_\delta,[A'_2,A''_2]_\alpha,[A''_2,A''_1]_\gamma,[A''_1,A'_1]_\alpha\}$,
$\{[A'_1,A'_2]_\alpha,[A'_2,B'_1]_\delta,[B'_1,B'_2]_\beta,[B'_2,A'_1]_\delta\}$,\\
$\{[B'_1,B'_2]_\delta,[B'_2,B''_2]_\beta,[B''_2,B''_1]_\gamma,[B''_1,B'_1]_\beta\}$,
$\{[A''_1,A''_2]_\alpha,[A''_2,B''_1]_\gamma,[B''_1,B''_2]_\beta,[B''_2,A''_1]_\delta\}$,\\
$\{[A'_1,A''_1]_\alpha,[A''_1,B''_2]_\gamma,[B''_2,B'_2]_\beta,[B'_2,A'_1]_\delta\}$ and
$\{[A'_2,A''_2]_\alpha,[A''_2,B''_1]_\gamma,[B''_1,B'_1]_\beta,[B'_1,A'_2]_\delta\}$.
\begin{figure}[t]
	\centering
	\includegraphics[width= 6cm]{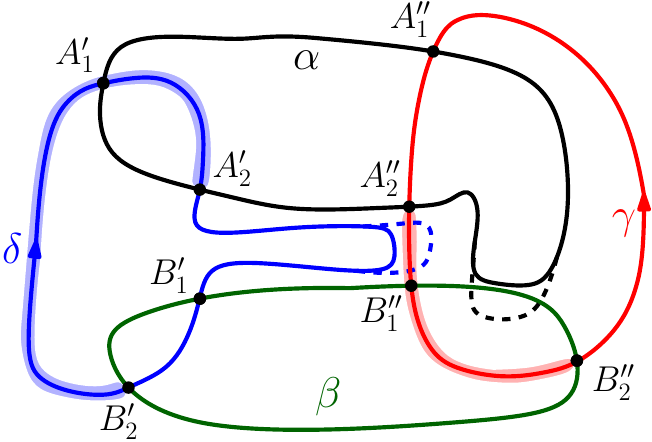}
	\caption{$\A(\{\alpha,\beta,\gamma,\delta\})$ after redrawing $\alpha$ and $\delta$ such that they do not form digons with $\beta$ and $\gamma$, respectively. If $d(\gamma^{\rm out},(\alpha^{\rm out},\beta^{\rm in}))=2$ and $d(\delta^{\rm in},(\alpha^{\rm out},\beta^{\rm in}))=1$, then $\gamma^{\rm out} \in (A''_2,B''_2)_\gamma$ and $\delta^{\rm in} \in (B'_2,A'_2)_\delta$.} 
	\label{fig:arrangement-of-4}
\end{figure}

Note that $d(\gamma^{\rm out},(\alpha^{\rm out},\beta^{\rm in}))=1$ and $d(\delta^{\rm in},(\alpha^{\rm out},\beta^{\rm in}))=2$ iff $\gamma^{\rm out} \in (B''_2,A''_2)_\gamma$ and $\delta^{\rm in} \in (A'_2,B'_2)_\delta$.
Similarly, $d(\gamma^{\rm out},(\alpha^{\rm out},\beta^{\rm in}))=2$ and $d(\delta^{\rm in},(\alpha^{\rm out},\beta^{\rm in}))=1$ iff $\gamma^{\rm out} \in (A''_2,B''_2)_\gamma$ and $\delta^{\rm in} \in (B'_2,A'_2)_\delta$.
Suppose for contradiction that one of these cases holds.
Then, on the one hand, there is a curve, namely $[\gamma^{\rm out},\delta^{\rm in}]_c$, which connects $\gamma^{\rm out}$ and $\delta^{\rm in}$ while not crossing any of $\alpha$, $\beta$, $\gamma$ and $\delta$ (this follows from the given cyclic order and Proposition~\ref{prop:cyclic-order}).
However, on the other hand, there is no face of $\A(\{\alpha,\beta,\gamma,\delta\})$ whose boundary contains segments of both $(B''_2,A''_2)_\gamma$ and $(A'_2,B'_2)_\delta$
or segments of both $(A''_2,B''_2)_\gamma$ and $(B'_2,A'_2)_\delta$, contradicting the existence of a curve as above.
\end{proof}

This completes the proof of Theorem~\ref{thm:main}.

\section{Concluding remarks}\label{sec:remarks}

What can we say about drawings on other surfaces? For example, we can draw $K_7$ on the torus, which implies the existence of 7 curves such that they pairwise form a digon, giving us 21 digons in total.  This suggests that we need a significantly different approach since an argument based on the embeddability of the double cover would give roughly $2n$ as an upper bound.  

\begin{problem}
    What is the maximum number of digons in a pairwise intersecting family of pseudocircles on a closed
orientable connected surface $M_g$ of genus g?
\end{problem}

We note that graphs with planar bipartite double covers were extensively studied in topological graph theory, see for example the works of Negami \cite{NE86,NE92}. Negami also considered general covers of graphs and formulated the following nice conjecture:

\begin{conjecture}[Negami 1988]
    A graph $G$ has a finite planar cover if and only if $G$ embeds in the projective plane.
\end{conjecture}

Indeed, with some extra work, one can show that digon graphs are embeddable in the projective plane, but we omit the proof. 

Another natural question is to consider the number of triangles in intersecting pseudocircle arrangements. Felsner and Scheucher \cite{FELSNER21} showed that the maximal number of triangles is $2n^2/3+ O(n)$, and this is asymptotically tight. On the other hand, triangles are unavoidable, they showed that there are at least $2n/3$ triangles and made the following conjecture.

\begin{conjecture}
Every intersecting arrangement of $n\ge 3$ pseudocircles has at least $n- 1$ triangles.    
\end{conjecture}

\nocite{*} 
\bibliographystyle{plain}
\bibliography{bibliography}

\end{document}